\documentclass[12pt,draftcls,onecolumn]{IEEEtran}
\usepackage[dvips]{epsfig}    % or this line, depending on which
\usepackage{mathptmx} % assumes new font selection scheme installed
\usepackage{amsmath} % assumes amsmath package installed
\usepackage{amssymb}  % assumes amsmath package installed
\DeclareMathAlphabet{\mathcal}{OMS}{cmsy}{m}{n}

\usepackage{pdfsync} % For sumatra
\usepackage[scriptsize]{subfigure}
\usepackage{color}
\usepackage{array}

\usepackage{enumerate}

\newtheorem{lemma}{Lemma}

\newtheorem{assumption}{Assumption}
\newenvironment{proof}{\textbf{Proof: }}{\hfill\rule{2mm}{2mm}}

\begin{document}
\title{Backstepping Transformation of Input Delay Nonlinear Systems}

\author{Delphine Bresch-Pietri 
\thanks{D. Bresch-Pietri is with the Department of Mechanical Engineering, Massachusetts Institute of Technology, Cambridge MA 02139, USA. Email address: \emph{dbp@mit.edu}}
and Miroslav Krstic
\thanks{M. Krstic is with the Department of Mechanical and Aerospace Engineering, University of California, San Diego, La Jolla CA 92093, USA. Email address: \emph{krstic@ucsd.edu}}
}
%\date{}
\maketitle
\begin{abstract}
We present here the details of a backstepping transformation aiming at reformulating the dynamics of a nonlinear systems subject to unknown long input delay in a form which is suitable for Lyapunov stability analysis. The control law underlying this transformation is predictor-based~\cite{artstein1982linear,kwon1980feedback,manitius1979finite}, as often considered for long delays. The proposed transformation follows recent results of the literature, based on the representation of the constant actuator delay as a transport Partial Differential Equation (PDE).
\end{abstract}

\section{Problem statement}
\label{section-problem}

Consider the following nonlinear plant
\begin{align}\label{plant-original}
	\dot X(t) =& f(X,U(t-D))
\end{align}
in which $X\in \mathbb{R}^n$, $f$ is a nonlinear function of class $\mathcal{C}^2$ such that $f(0,0) = 0$, $U$ is scalar and $D$ is an unknown delay belonging to a known interval $[\underline D,\overline D]$ (with $\underline D>0$).

\begin{assumption}\label{assumption-SFC}
	The plant $\dot X = f(X,\Omega)$ is strongly forward complete.
\end{assumption}

\begin{assumption}\label{assumption-exp-stab-global}
	There exists a feedback law $U(t) = \kappa(X(t))$ such that the nominal delay-free plant is globally exponentially stable and such that $\kappa$ is a class $\mathcal{C}^2$ function, i.e. there exist (see resp. Theorem 4.14 and Theorem 2.207 in \cite{khalil2002nonlinear, pralyfonctions}) $\lambda>0$ and a class $\mathcal{C}^{\infty}$ radially unbounded positive definite function $V$ such that for $x\in \mathbb{R}^n$
\begin{align}
	\frac{d V}{d X}(X) f(X,\kappa(X)) \leq& -\lambda V(X) \\
	|X|^2 \leq V(X) \leq& c_1 |X|^2 \\
	\left|\frac{d V}{d X}(X) \right| \leq& c_2 |X|
\end{align}
for given $c_1,c_2>0$.
\end{assumption}

Assumption~\ref{assumption-SFC} guarantees that~\eqref{plant-original} does not escape in finite time and, in particular, before the input reaches the system at $t=D$. This is a reasonable assumption to enable stabilization. The difference from the standard notion of forward completeness \cite{angeli1999forward} comes from the fact that we assume that $f(0,0) = 0$. Assumption~\ref{assumption-exp-stab-global} guarantees that the delay-free plant is (globally) exponentially stabilisable.

To analyze the closed-loop stability despite delay uncertainties, we use the systematic Lyapunov tools introduced in~\cite{krstic2009delay_book} and first reformulate plant~\eqref{plant-original} in the form
\begin{align} \label{plant-u}
	\left\{ \begin{array}{rcl}
		\dot X(t) &=& f(X(t),u(0,t)) \\
		D u_t (x,t) &=& u_x (x,t) \\
		u(1,t) &=& U(t)
	\end{array}\right.
\end{align}
by introducing the following distributed input
\begin{align}\label{def-u}
	u(x,t) = U(t+D(x-1)) \,, \quad x\in[0,1]
\end{align}
In details, the input delay is now represented as a coupling with a transport PDE driven by the input and with unknown convection speed $1/D$. We now propose to reformulate this plant thanks to a backstepping transformation of the (estimated) distributed input to obtain a dynamics compliant with Laypunov analysis,

\section{Backstepping transformation for unmeasured distributed input}
\label{section-est}

In this paper, we consider the actuator state $u(\cdot,t)$ to be unmeasured, as is typically the case in applications. To deal with this fact, we introduce a distributed input estimate
\begin{align}
	\label{def-uhat}
	\hat u(x,t) =& U(t + \hat D(t) (x-1)) \,,\quad x \in [0,1]
\end{align}
Applying the certainty equivalence principle to the nominal dynamics (i.e. from the case of a known input delay), the control law is chosen as
\begin{align}\label{control-est}
	U(t) =& \kappa(\hat p(1,t))
\end{align}
in which the distributed predictor estimate is defined in terms of the actuator state estimate as
\begin{align}\label{def-phat}
	\hat p(x,t) =& X(t + \hat D(t) x) 
	= X(t) + \hat D(t) \int_0^x f(\hat p(y,t), \hat u(y,t)) d  y
\end{align}
and the delay estimate $\hat D$ is a time-differentiable function.

\begin{lemma}
	The backstepping transformation of the distributed input estimate~\eqref{def-uhat}
	\begin{align}\label{def-what}
		\hat w(x,t) =& \hat u(x,t) - \kappa(\hat p(x,t)) \,,
	\end{align}
	in which the distributed predictor estimate is defined in~\eqref{def-phat}, together with the control law~\eqref{control-est}, transforms plant~\eqref{plant-u} into
	\begin{align}
		\label{eq-X-est}
		\dot X(t) =& f(X(t),\kappa(X(t) + \hat w(0,t) + \tilde u(0,t)) \\
		\label{PDE-what}
		\hat D(t) \hat w_t =& \hat w_x + \dot{\hat D}(t) q_1(x,t) 
		- q_2(x,t) f_{\tilde u}(t)\\
		\label{BC-what1}
		\hat w(1,t) =& 0 \\
		\label{PDE-utilde}
		D \tilde u_t =& \tilde u_x - \tilde D(t) p_1(x,t) 
					- \dot{\hat D}(t) p_2(x,t) \\
		\label{BC-utilde1}
		\tilde u(1,t) =& 0
\end{align}
in which 
\begin{align}\label{def-utilde}
	\tilde u(x,t) = u(x,t)-\hat u(x,t)
\end{align}
is the distributed input estimation error and
\begin{align}
	p_1&(x,t) = \frac{D}{\hat D(t)} \left[\hat w_x (x,t)
		+ \hat D(t)\frac{d\kappa}{d\hat p} (\hat p(x,t)) 
		f(\hat p(x,t),\hat w(x,t)+\kappa(\hat p(x,t))) \right]\\
	p_2&(x,t) = \frac{D}{\hat D(t)} (x-1) \left[\hat w_x (x,t)
		+ \hat D(t)\frac{d\kappa}{d\hat p} (\hat p(x,t)) 
		f(\hat p(x,t),\hat w(x,t)+\kappa(\hat p(x,t))) \right]\\
	q_1&(x,t) = (x-1) \left[\hat w_x (x,t)
		+ \hat D(t)\frac{d\kappa}{d\hat p} (\hat p(x,t)) 
		f(\hat p(x,t),\hat w(x,t)+\kappa(\hat p(x,t))) \right]\nonumber\\
		& - \hat D(t)\frac{d\kappa}{d \hat p}(\hat p(x,t))
		\int_0^x \Phi(x,y)\left[ 
		f(\hat p(y,t),\hat w(y,t)+\kappa(\hat p(y,t))) 
		+ \frac{\partial f}{\partial \hat u} (\hat p(y,t), 
		\hat w(y,t)+\kappa(\hat p(y,t))) \right.\nonumber\\
		&\left.\times (y-1)  \left[\hat w_x(y,t) 
		+ \hat D(t)\frac{d\kappa}{d\hat p} (\hat p(y,t)) 
		f(\hat p(y,t),\hat w(y,t)+\kappa(\hat p(y,t))) \right] \right] d y 
\end{align}
\begin{align}
	q_2&(x,t) = \hat D(t)\frac{d\kappa}{d \hat p}(\hat p(x,t)) \Phi(x,0) \\
	f_{\tilde u}&(t) = f(\hat p(0,t),	u(0,t))-f(\hat p(0,t),\hat u(0,t))
\end{align}
where $\Phi$ is the transition matrix associated with the space-varying time-parametrized equation $\frac{d r}{d x}(x) = \hat D(t) \frac{\partial f}{\partial \hat p}(\hat p(x,t),\hat w(x,t) + \kappa(\hat p(x,t))) r(x)$.
\end{lemma}

\begin{proof}
First, Eq.~\eqref{eq-X-est} can be directly obtained from definitions~\eqref{def-u},~\eqref{def-what} and the one of $\tilde u$. Second, one can easily obtain from~\eqref{def-uhat} that the estimate distributed input satisfies
\begin{align}
	\label{PDE-uhat}
	\hat D(t) \hat u_t(x,t) =& \hat u_x(x,t) + \dot{\hat D}(t)(x-1)\hat u_x(x,t)
	\\
	\label{BC-uhat1}
	\hat u(1,t) = U(t)
\end{align}
Matching this equation with~\eqref{plant-u} gives~\eqref{PDE-utilde} and~\eqref{BC-utilde1}, in which we have used~\eqref{def-what} to express the functions $p_1$ and $p_2$ in terms of $\hat w$ and $\hat w_x$. Before studying the governing equation of the distributed input, we focus on the dynamics of the distributed predictor. The temporal and spatial derivative of $\hat p(x,t)$ can be expressed as follows
\begin{align}
	\hat p_t =& f(\hat p(0,t), u(0,t)) 
		+ \dot{\hat D}(t) \int_0^x f(\hat p(y,t),
		\hat u(y,t)) d y \nonumber\\
		& + \hat D(t) \int_0^x \left[ \frac{\partial f}
		{\partial \hat p}(\hat p(y,t),\hat u(y,t)) \hat p_t(y,t) 
		+ \frac{\partial f}{\partial \hat u}(\hat p(y,t),\hat u(y,t)) 
		\hat u_t(y,t)\right] d y \\
	\hat p_x =& 
		\hat D(t) f(\hat p(0,t),\hat u(0,t)) +  \hat D(t) \int_0^x \left[
		\frac{\partial f}{\partial \hat p}(\hat p(y,t),\hat u(y,t))
		\hat p_x(y,t) + \frac{\partial f}{\partial \hat u}
		(\hat p(y,t),\hat u(y,t)) \hat u_x(y,t)
	\right] d y 
\end{align}
Therefore, using the governing equation of the distributed input estimate given in~\eqref{PDE-uhat},
\begin{align}
	\hat D(t) \hat p_t(x,t) - \hat p_x (x,t) =&
	\hat D(t) \left[f(\hat p(0,t),u(0,t))-f(\hat p(0,t),\hat u(0,t))\right]
		+ \dot{\hat D}(t) \hat D(t) \int_0^x f(\hat p(y,t),\hat u(y,t)) d y 
		\nonumber\\
		& + \hat D(t) \int_0^x \frac{\partial f}{\partial \hat p}
	(\hat p(y,t),\hat u(y,t)) [\hat D(t) \hat p_t(y,t) - \hat p_x(y,t)] d y
	\nonumber\\ &
	+ \dot{\hat D}(t) \hat D(t) \int_0^x \frac{\partial f}{\partial \hat u}
		(\hat p(y,t),\hat u(y,t)) (y-1) \hat u_x(y,t) d y
\end{align}
Consider a given $t\geq 0$ and denote $r(x) =  \hat D(t) \hat p_t(x,t) - \hat p_x(x,t)$. Taking a spatial derivative of the latter equality, one can obtain the following equation in $x$, parametrized in $t$,
\begin{align}
\left\{ \begin{aligned}
	\frac{d r}{d x}(x) =& \hat D(t) \frac{\partial f}{\partial \hat p}
		(\hat p(x,t), \hat u(x,t)) r(x) + \dot{\hat D}(t) \hat D(t) \left[ 
		f(\hat p(x,t),\hat u(x,t)) + \frac{\partial f}{\partial \hat u} 
		(\hat p(x,t), \hat u(x,t)) (x-1) \hat u_x(x,t) \right] \\
	r(0) =& \hat D(t) \left[f(\hat p(0,t),u(0,t))-
		f(\hat p(0,t),\hat u(0,t))\right]
\end{aligned}\right.
\end{align}
Defining the transition matrix $\Phi$ associated to the corresponding homogeneous equation, one can solve this equation and obtain
\begin{align}
	\label{PDE-phat}
	\hat D(t) \hat p_t =&\hat p_x + \Phi(x,0,t) \hat D(t) \left[f(\hat p(0,t),
		u(0,t))-f(\hat p(0,t),\hat u(0,t))\right] \nonumber\\
		& + \dot{\hat D}(t) \hat D(t) \int_0^x \Phi(x,y,t)\left[ 
		f(\hat p(y,t),\hat u(y,t)) + \frac{\partial f}{\partial \hat u} 
		(\hat p(y,t), \hat u(y,t)) (y-1) \hat u_x(y,t) \right] d y
\end{align}
Now, matching the time- and space-derivatives of the backstepping transformation~\eqref{def-what}
\begin{align*}
	\hat w_t(x,t) =& \hat u_t(x,t) - \frac{d\kappa}{d\hat p}(\hat p(x,t))
		\hat p_t(x,t) \\
	\hat w_x(x,t) =& \hat u_x(x,t) - \frac{d\kappa}{d\hat p}(\hat p(x,t))
		\hat p_x(x,t)
\end{align*}
with the governing equations~\eqref{PDE-uhat} and~\eqref{PDE-phat}, one can obtain~\eqref{PDE-what} and use the backstepping transformation~\eqref{def-what} to express the functions $q_1$ and $q_2$ in terms of $\hat w$ and its spatial-derivative.
\end{proof}

Comparing~\eqref{eq-X-est}-\eqref{BC-utilde1} to plant~\eqref{plant-u}, one can see that the main advantage of this new representation is that the boundary conditions~\eqref{BC-what1} and~\eqref{BC-utilde1} are now equal to zero, consistently with the choice of the control law~\eqref{control-est}, as opposed to the one stated in~\eqref{plant-u}. This is particularly for stability analysis.

To provide a total description of the system dynamics, we also need the governing equation of spatial derivatives of the distributed variables, which are given in the following lemma.
\begin{lemma}\label{lemma-der}
The spatial derivatives of the distributed input estimation error~\eqref{def-utilde} and of the backstepping transformation~\eqref{def-what} satisfy
\begin{align}\label{PDE-utildex}
	\left\{ \begin{aligned}
		D \tilde u_{x t} =& \tilde u_{xx} - \tilde D(t) p_3(x,t) 
			- \dot{\hat D}(t) p_4(x,t) \\
		\tilde u_x(1,t) =& \tilde D(t)p_1(1,t)
	\end{aligned}\right.
\end{align}
\begin{align}\label{PDE-whatx}
	\left\{ \begin{aligned}
		\hat D(t) \hat w_{x t} =& \hat w_{xx} + \dot{\hat D}(t) q_3(x,t) 
		- q_4(x,t) f_{\tilde u}(t)\\
	\hat w_x(1,t) =& -\dot{\hat D}(t) q_1(1,t) + q_2(1,t) f_{\tilde u}(t)
	\end{aligned}\right.
\end{align}
\begin{align}\label{PDE-whatxx}
	\left\{ \begin{aligned}
		\hat D(t) \hat w_{x x t} =& \hat w_{x x x} + \dot{\hat D}(t) q_5(x,t) 
		- q_6(x,t) f_{\tilde u}(t)\\
	\hat w_{x x}(1,t) =& -\dot{\hat D}(t) q_3(1,t) + q_4(1,t) f_{\tilde u}(t)
		+ \hat D(t) q_7(t)
	\end{aligned}\right.
\end{align}
in which $p_3, p_4,q_3,q_4,q_5,q_6$ and $q_7$ are given in Appendix.
\end{lemma}

\begin{proof}
Taking a spatial derivative of~\eqref{PDE-utilde}, one can obtain the governing equation in~\eqref{PDE-utildex} and, from the boundary condition~\eqref{BC-utilde1}, that $\tilde u_t(1,t) = 0$ which gives, replacing in~\eqref{PDE-utilde}, the boundary condition in~\eqref{PDE-utildex}. The exact same arguments applied to~\eqref{PDE-what}-\eqref{BC-what1} governing the backstepping transformation give system~\eqref{PDE-whatx}.

Taking a spatial derivative of the first equation in~\eqref{PDE-whatx} give the one in~\eqref{PDE-whatxx}. Finally, using the first equation in~\eqref{PDE-whatx} for $x=1$, one can obtain
\begin{align}
	\hat w_{xx}(1,t) =& - \dot{\hat D}(t) q_3(x,t) + q_4(1,t) f_{\tilde u}(t)
		+ \hat D(t) \hat w_{x t}(1,t)
\end{align}
in which $\hat w_{x,t}(1,t) = q_7(t)$ can be reformulated by taking a time derivative of the boundary condition in~\eqref{PDE-whatx}.
Finally, the functions $p_3, p_4,q_3,q_4,q_5,q_6$ and $q_7$ given in Appendix can be expressed in terms of $\hat w(\cdot,t)$ and its spatial derivative by using the backtespping transformation~\eqref{def-what} and its spatial derivative versions.
\end{proof}

\section{Conclusion}

In this paper, we have presented a backstepping transformation aiming at reformulating the dynamics of a nonlinear systems subject to unknown long input delay in a form which is suitable for Lyapunov stability analysis. This transformation will be particularly useful in future works to perform a Lyapunov analysis of closed-loop stability to delay uncertainties.

\appendix

\subsection{Expression of the functions involved in Lemma~\ref{lemma-der}}
\vspace{-1cm}
\begin{align}
	p_3(x,t) =& p_{1,x}(x,t) \nonumber\\
		=&\frac{D}{\hat D(t)} \bigg[\hat w_{xx} 
			+ \hat D(t) \frac{d\kappa}{d\hat p} (\hat p(x,t))\frac{d}{ d x}
			\left[f(\hat p(x,t), \hat w(x,t) + \kappa(\hat p(x,t))) \right] 
			\nonumber \\
			+ \hat D(t)^2 &f(\hat p(x,t),\hat w(x,t) + \kappa(\hat p(x,t)))^T
				\frac{d^2\kappa}{d\hat p^2} (\hat p(x,t)) f(\hat p(x,t),\hat w(x,t) 
				+ \kappa(\hat p(x,t))) \bigg]\\
	p_4(x,t) =& \hat p_{2,x}(x,t) \nonumber\\
		=& \frac{D}{\hat D(t)}\bigg[\hat w_x+ \hat D(t)
			\frac{d\kappa}{d\hat p}(\hat p(x,t)) f(\hat p(x,t),\hat w(x,t)
			+\kappa(\hat p(x,t)))	+ (x-1) \bigg[ \hat w_{xx}(x,t) \nonumber\\
			&+ \hat D(t)^2 f(\hat p(x,t),\hat w(x,t) + \kappa(\hat p(x,t)))^T
				\frac{d^2\kappa}{d\hat p^2}(\hat p(x,t)) f(\hat p(x,t),\hat w(x,t) 
				+ \kappa(\hat p(x,t))) \nonumber\\
			&+ \hat D(t) \frac{d\kappa}{d\hat p}(\hat p(x,t)) 
			\frac{d}{ d x} \left[f(\hat p(x,t),
			\hat w(x,t) + \kappa(\hat p(x,t))) \right] \bigg] \bigg]
\end{align}
\begin{align}
	q_3(x,t) =& q_{1,x}(x,t)\nonumber\\
	=& \hat w_x(x,t) + \hat D(t)\frac{d\kappa}{d\hat p}(\hat p(x,t))
		f(\hat p(x,t),\hat w(x,t) + \kappa(\hat p(x,t))) 
		+ (x-1)\bigg[\hat w_{xx}(x,t) \nonumber\\
	&+ \hat D(t) \frac{d\kappa}{d\hat p} (\hat p(x,t))\frac{d}{ d x}
		\left[f(\hat p(x,t), \hat w(x,t) + \kappa(\hat p(x,t))) \right] 
		\nonumber \\
	&+ \hat D(t)^2 f(\hat p(x,t),\hat w(x,t) + \kappa(\hat p(x,t)))^T
		\frac{d^2\kappa}{d\hat p^2} (\hat p(x,t)) f(\hat p(x,t),\hat w(x,t) 
		+ \kappa(\hat p(x,t))) \bigg] \nonumber\\
	&- \hat D(t)
		\frac{d\kappa}{d \hat p}(\hat p(x,t)) \bigg[
		f(\hat p(x,t),\hat w(x,t)+\kappa(\hat p(x,t))) 
		+ \frac{\partial f}{\partial \hat u} 
		(\hat p(x,t), \hat w(x,t)+\kappa(\hat p(x,t))) \nonumber\\
	& \times (x-1) \bigg[\hat w_x(x,t) 
		+ \hat D(t)\frac{d\kappa}{d\hat p}(\hat p(x,t))
		f(\hat p(x,t),\hat w(x,t) + \kappa(\hat p(x,t))) \bigg] \bigg] \\
	q_4(x,t) =& q_{2,x}(x,t) \nonumber\\
		=& \hat D(t)^2 \bigg[ f(\hat p(x,t),\hat w(x,t)+\kappa(\hat p(x,t)))^T
		\frac{d^2 \kappa}{d \hat p^2 }(\hat p(x,t))
		f(\hat p(x,t),\hat w(x,t)+\kappa(\hat p(x,t))) \nonumber\\
	&	+ \frac{d\kappa}{d \hat p}(\hat p(x,t))
		\frac{\partial f}{\partial \hat p}(\hat p(x,t),	
		\hat w(x,t)+\kappa(\hat p(x,t))) \bigg]
		\Phi(x,0,t) \\
	q_5(x,t) =& q_{3,x}(x,t)\nonumber\\
	=&2\bigg[\hat w_{xx}(x,t) 
		+ \hat D(t) \frac{d\kappa}{d\hat p} (\hat p(x,t))\frac{d}{ d x}
		\left[f(\hat p(x,t), \hat w(x,t) + \kappa(\hat p(x,t))) \right] 
		\nonumber \\
	&+ \hat D(t)^2 f(\hat p(x,t),\hat w(x,t) + \kappa(\hat p(x,t)))^T
		\frac{d^2\kappa}{d\hat p^2} (\hat p(x,t)) f(\hat p(x,t),\hat w(x,t) 
		+ \kappa(\hat p(x,t))) \bigg] \nonumber \\
	&+ (x-1) \bigg[\hat w_{xxx}(x,t) +\frac{d}{dx}\bigg[
		+ \hat D(t) \frac{d\kappa}{d\hat p} (\hat p(x,t))\frac{d}{ d x}
		\left[f(\hat p(x,t), \hat w(x,t) + \kappa(\hat p(x,t))) \right] 
		\nonumber \\
	&+ \hat D(t)^2 f(\hat p(x,t),\hat w(x,t) + \kappa(\hat p(x,t)))^T
		\frac{d^2\kappa}{d\hat p^2} (\hat p(x,t)) f(\hat p(x,t),\hat w(x,t) 
		+ \kappa(\hat p(x,t))) \bigg] \bigg] \nonumber\\
	&- \hat D(t)
		\frac{d}{d x} \left(\frac{d\kappa}{d \hat p}(\hat p(x,t)) \bigg[
		f(\hat p(x,t),\hat w(x,t)+\kappa(\hat p(x,t))) 
		+ \frac{\partial f}{\partial \hat u} 
		(\hat p(x,t), \hat w(x,t)+\kappa(\hat p(x,t))) \right. \nonumber\\
	&\left.\times (x-1) \bigg[ \hat w_x(x,t) + \hat D(t) \frac{d\kappa}{d\hat p}
		(\hat p(x,t)) f(\hat p(x,t),\hat w(x,t) + \kappa(x,t)) \bigg] \right) \\
	q_6(x,t) =& q_{4,x}(x,t) \nonumber\\
		=& \hat D(t)^2 \frac{d}{d x} \bigg[ 
		f(\hat p(x,t),\hat w(x,t)+\kappa(\hat p(x,t)))^T 
		\frac{d^2 \kappa}{d \hat p^2 }(\hat p(x,t)) 
		f(\hat p(x,t),\hat w(x,t)+\kappa(\hat p(x,t))) \nonumber\\
	&+ \frac{d\kappa}{d \hat p}(\hat p(x,t))
		\frac{\partial f}{\partial \hat p}(\hat p(x,t),	
		\hat w(x,t)+\kappa(\hat p(x,t))) \bigg]
		\Phi(x,0,t) \nonumber\\
		& + \hat D(t)^3 \left[f(\hat p(x,t),\hat w(x,t)+\kappa(\hat p(x,t)))^T 
		\frac{d^2 \kappa}{d \hat p^2 }(\hat p(x,t))
		f(\hat p(x,t),\hat w(x,t)+\kappa(\hat p(x,t))) \right. \nonumber
		%\\
\end{align}
\begin{align}
	&\left.+ \frac{d\kappa}{d \hat p}(\hat p(x,t))
		\frac{\partial f}{\partial \hat p}(\hat p(x,t),	
		\hat w(x,t)+\kappa(\hat p(x,t))) \right]
		\frac{\partial f}{\partial \hat p}(\hat p(x,t),
		\hat w(x,t)+\kappa(\hat p(x,t)))\Phi(x,0,t) \\
	q_7(t) =& - \ddot{\hat D}(t)q_1(1,t) - \dot{\hat D}(t)q_{1,t}(1,t)
		+ \dot{\hat D}(t) \frac{\partial \kappa}{\partial \hat p}(\hat p(x,t))
		\Phi(x,0,t) f_{\tilde u}(t) \nonumber\\&
		+ \hat D(t) \hat p_t(x,t)^T 
		\frac{\partial^2 \kappa}{\partial \hat p^2}(\hat p(x,t))
		\Phi(x,0,t) f_{\tilde u}(t) \nonumber\\
	& + q_2(1,t) f_{d p}(t) f(\hat p(0,t),\tilde u(0,t) +\hat w(0,t)
		+\kappa(\hat p(0,t))) 
		+ q_2(1,t)f_{d u} (t) \frac{\tilde u_x(0,t) + \hat u_x(0,t)}{D} 
		\nonumber\\
	& + q_2(1,t) \frac{\partial f}{\partial\hat u}
		(\hat p(0,t),\hat w(0,t) + \kappa(\hat p(0,t))) 
		\frac{\tilde u_x(0,t) - \tilde D(t) p_1(0,t) 
		- \dot{\hat D}(t) p_2(0,t)}{D} 
\end{align}
and in which we have used
\begin{align}
	f_{d p}(t) =& 
		\frac{\partial f}{\partial \hat p}(\hat p(0,t),u(0,t))
		- \frac{\partial f}{\partial \hat p}(\hat p(0,t),\hat u(0,t))\\
	f_{d u}(t) =& 
		\frac{\partial f}{\partial \hat u}(\hat p(0,t),u(0,t))
		- \frac{\partial f}{\partial \hat u}(\hat p(0,t),\hat u(0,t))\\
	q_{1,t}(1,t) =& - \dot{\hat D}(t)\frac{d\kappa}{d \hat p}(\hat p(1,t))
		\int_0^1 \Phi(1,y,t)\bigg[ 
		f(\hat p(y,t),\hat w(y,t)+\kappa(\hat p(y,t))) \nonumber\\
	&+ \frac{\partial f}{\partial \hat u} 
		(\hat p(y,t), \hat w(y,t)+\kappa(\hat p(y,t))) (y-1) 
		\bigg[ \hat w_x(y,t) + \hat D(t) \frac{d\kappa}{d\hat p}(\hat p(y,t))
		\nonumber\\
	&\times
		f(\hat p(y,t),\hat w(y,t) + \kappa(\hat p(y,t)))\bigg] \bigg] d y 
		- \hat D(t) \hat p_t(1,t)^T \frac{d^2\kappa}{d \hat p^2}(\hat p(1,t))
		\int_0^1 \Phi(1,y,t) \nonumber \\
	&\times \left[ 
		f(\hat p(y,t),\hat w(y,t)+\kappa(\hat p(y,t))) 
		+ \frac{\partial f}{\partial \hat u} 
		(\hat p(y,t), \hat w(y,t)+\kappa(\hat p(y,t))) (y-1) 
		\bigg[ \hat w_x(y,t) \right.\nonumber\\
	&\left.+ \hat D(t) \frac{d\kappa}{d\hat p}(\hat p(y,t))
		f(\hat p(y,t),\hat w(y,t) + \kappa(\hat p(y,t)))\bigg] \right] d y 			
		\nonumber\\
	&- \hat D(t) \frac{d\kappa}{d \hat p}(\hat p(1,t))
		\int_0^1 \Phi(1,y,t)\bigg[ 
		\frac{\partial f}{\partial \hat p}(\hat p(y,t),
		\hat w(y,t)+\kappa(\hat p(y,t))) \hat p_t(y,t) \nonumber\\
	&+\frac{\partial f}{\partial \hat u}(\hat p(y,t),
		\hat w(y,t)+\kappa(\hat p(y,t)))\hat u_t(y,t)
		+ \bigg[ \frac{\partial^2 f}{\partial \hat u \partial \hat p} 
		(\hat p(y,t), \hat w(y,t)+\kappa(\hat p(y,t))) \hat p_t(y,t) \nonumber\\
	&+ \frac{\partial^2 f}
		{\partial \hat u^2} (\hat p(y,t), \hat w(y,t)+\kappa(\hat p(y,t))) 
		\hat u_t(y,t)\bigg] (y-1) \bigg[ \hat w_x(y,t) + \hat D(t) 
		\frac{d\kappa}{d\hat p}(\hat p(y,t)) \nonumber\\
	&\times f(\hat p(y,t),\hat w(y,t) + \kappa(\hat p(y,t)))\bigg] 
		+ \frac{\partial f}{\partial \hat u}(\hat p(y,t), 
		\hat w(y,t)+\kappa(\hat p(y,t)))
		(y-1) \hat u_{x t}(y,t)\bigg] d y
\end{align}
with
\begin{align}
	\hat p_t(x,t) =& \frac{1}{\hat D(t)} \bigg[
		\hat p_x(x,t) + \Phi(x,0,t) \hat D(t) f_{\tilde u}(t) \nonumber\\
	& + \dot{\hat D}(t) \hat D(t) \int_0^x \Phi(x,y,t)\left[ 
		f(\hat p(y,t),\hat u(y,t)) + \frac{\partial f}{\partial \hat u} 
		(\hat p(y,t), \hat u(y,t)) (y-1) \bigg[ \hat w_x(y,t) \right.\nonumber\\
	&\left.+ \hat D(t) \frac{d\kappa}{d\hat p}(\hat p(y,t))
		f(\hat p(y,t),\hat w(y,t) + \kappa(\hat p(y,t)))\bigg]\right] d y \bigg]\\
	\hat u_t(x,t) =& \frac{1 + \dot{\hat D}(t)(x-1)}{\hat D(t)} \bigg[
		\hat w_x(x,t) + \hat D(t) \frac{d\kappa}{d\hat p}(\hat p(x,t))
		f(\hat p(x,t),\hat w(x,t) + \kappa(\hat p(x,t)))\bigg]
\end{align}
\begin{align}
	\hat u_{x t}(x,t) =& \frac{1}{\hat D(t)} \bigg[
		(1+\dot{\hat D}(t)) \bigg[\hat w_x(x,t) 
		+ \hat D(t) \frac{d\kappa}{d\hat p}(\hat p(x,t))
		f(\hat p(x,t),\hat w(x,t)+\kappa(\hat p(x,t))) \bigg] 
		+ \dot{\hat D}(t)(x-1)\nonumber\\
	& \times \bigg[\hat w_{xx}(x,t) 
		+ \frac{d\kappa}{d\hat p}(\hat p(x,t) \frac{d}{d x}\bigg[
		f(\hat p(x,t),\hat w(x,t) + \kappa(\hat p(x,t))) \bigg]	\nonumber\\
	&+ \hat D(t)^2 f(\hat p(x,t),\hat w(x,t) +\kappa(\hat p(x,t)))^T
		\frac{d^2\kappa}{d\hat p^2}(\hat p(x,t))
		f(\hat p(x,t),\hat w(x,t)+\kappa(\hat p(x,t))) \bigg]
	\end{align}

\bibliographystyle{plain}
\bibliography{../../../Biblio/biblioDBP}

\end{document}